\documentclass[12pt]{article} 
\usepackage{amsthm, amsmath, amssymb, amsfonts, graphicx, epsfig, bbm}
\usepackage{mathrsfs}
\usepackage{accents, comment}
\usepackage[numbers]{natbib}

\usepackage{amsfonts}
\usepackage[english]{babel}

\usepackage[T1]{fontenc}
\usepackage[color=white]{todonotes}

\newtheorem{theorem}{Theorem}[section]
\newtheorem{lemma}[theorem]{Lemma}
\newtheorem{proposition}[theorem]{Proposition}

\newtheorem{remark}[theorem]{Remark}

\numberwithin{equation}{section}
\usepackage{fancyhdr}
\usepackage{lastpage}
\renewcommand{\P}{\mathbb{P}}
\usepackage{color}
\usepackage[normalem]{ulem}

\newcommand{\E}{\mathbb{E}} 
\newcommand{\R}{\mathbb{R}} 

\DeclareMathOperator{\Var}{Var}	

\title{On small deviations of Gaussian multiplicative chaos with a strictly logarithmic covariance on Euclidean ball}

\author{\begin{tabular}{ll} Anna Talarczyk  &Maciej Wiśniewolski\footnote{Supported by University of Warsaw grant IDUB-622-239/2022 and IDUB-POB3-D110-003/2022}\\
         {\tt annatal@mimuw.edu.pl}& {\tt  wisniewolski@mimuw.edu.pl 
} 
        \end{tabular}
}

\begin{document}

\maketitle

\begin{center}
{\small
 Institute of Mathematics, University
of Warsaw \\
  Banacha 2, 02-097 Warszawa, Poland 
}

\end{center}
\begin{abstract} Recognizing the regime of positive definiteness for a strictly logarithmic covariance kernel, we prove that the small deviations of a related Gaussian multiplicative chaos (GMC) $M_\gamma$  are   for each natural dimension $d$ always of lognormal type,
 i.e. the upper and lower limits as $t\to \infty$ of
$$
	-\ln\Big(\mathbb{P}(M_\gamma(B(0,r))\le \delta \Big)/(\ln \delta)^2  
$$
 are finite and bounded away from zero.  We then place the small deviations in the context of Laplace transforms of $M_\gamma$ and discuss the explicit bounds on the associated constants. 
We also provide some new representations of the Laplace transform of GMC related to a strictly logarithmic covariance kernel.
\end{abstract}
 
\noindent
\begin{quote}


\end{quote}

\section{Introduction}

We study the small deviations and Laplace transforms of Gaussian multiplicative chaos (GMC) on Euclidean ball in $\mathbb{R}^d$, in case the kernel of associated Gaussian field is strictly logarithmic.
Recall that if $Z$ is a generalized Gaussian random field on $D\subset \R^d$
with covariance of the form 
\begin{equation*}
 C(x,y)=-\ln|x-y|+f(x,y)
\end{equation*}
where $f$ is bounded, then
GMC on a borel subset $B$ of  $D$ is formally defined as
\begin{equation}
 M_{\gamma}(B)=\int_B e^{\gamma Z(x)-\frac{\gamma^2}2EZ^2(x)}dx
\label{e:def_M_gamma}
 \end{equation}
where $\gamma<\sqrt {2d}$. In our situation $f$ will be a constant. One gives a meaning to the above expression by an appropriate limiting procedure (see e.g. \cite{Ber1}).
 The importance of the theory of GMC
comes from various fields especially theoretical physics (see \cite{RV1}, \cite{Kup}).

 The lognormal type of the behavior of GMC in the vicinity of $0$  means that there exist positive constants $c_B, C_B$, $C_1 (B), C_2(B)$,  such that for small values of $\delta$ we have
\begin{align}
	C_1(B) e^{-c_B (\ln\delta)^2 } \le \mathbb{P}(M_\gamma(B) \leq \delta)\le C_2(B) e^{-C_B (\ln\delta)^2 }, 
\label{e:two_sided_estimates}
	\end{align}
up to a multiplicative constant and terms of higher rate of decay. The last identity can be equivalently formulated by the means of Laplace transform and its behavior at $+\infty$.   

 In their seminal work Duplantier and Sheffield \cite{Dupl} found the upper lognormal bound on the small deviations of GMC  in case $d = 2$ corresponding to the logarithmic covariance of the Gaussian free field,  \cite{Dupl} (see also \cite{Aru}).
In case $d = 1$ Remy and Zhu \cite{Remy} determined the explicit form of the distribution of GMC proving in particular the lognormal behavior of small deviations in case the spatial mean of the field does not vanish. In case $d = 2$ Lacoin, Rhodes and Vargas \cite{Lac} found the exact rate of decay of GMC in the vicinity of $0$, faster than lognormal, if spatial mean of the associated Gaussian field vanishes, that is, if $\int_{D}Z(x)dx\equiv 0$, where $D$ denotes a domain.
They conclude that methodology herein works in any higher dimension. Although their results confirm earlier intuitions that lognormal behavior of GMC is somehow connected to the spatial mean not vanishing, we will show in the sequel that in case of strictly logarithmic  covariance kernel, to understand the small deviations on a ball we need to refer to  the threshold of positive definiteness.

Even if we decompose the field $Z$ into the spatial mean and its $0$--mean counterpart, obtained components are not independent and it is not immediate to conclude the lognormal type of related small deviations. Moreover, we ask if lognormality is related to the threshold of positive definiteness.  
 Above the threshold of positive definiteness the situation is relatively simple, 
 since the lognormal type of small deviations can be deduced from the scaling property. This argument cannot be used in the borderline case, that is, when we consider GMC on a maximal ball on which $C(x,y)$ is positive definite.
  A methodology to give a lower bound on the probability of small values of GMC  is  presented in  \cite[Ch.3, Ex.7]{Ber1}, where lognormality is claimed by Karhuhen–Loeve expansions and the fact the smallest of associated eigenvalues is strictly positive. However, the last property should be carefully verified making actually the claim in \cite[Ch.3, Ex.7]{Ber1} not true in a general case (also by the case of vanishing spatial mean). We will give complete description in the case of the strictly logarithmic kernel, using the notion of logarithmic potential and referring the problem to the threshold of positive definiteness.   
Last but not least, in all of the above mentioned papers, the constants $c_B$ and $C_B$ of the above estimates, are not explicit.

 The main objective of the article is to classify the small deviations of GMC on Euclidean balls $B(0,r)$ related to strictly logarithmic covariance kernel, i.e. 
\begin{equation}
 C(x,y)=\ln \frac {T}{|x-y|},
\label{e:covC}
 \end{equation}
where $T>0$ is such that $C$ is positive definite at least in the unit ball $B(0,1)\subset \R^d$. It is known that if $T$ is large enough, then \eqref{e:covC} is positive definite on $B(0,1)$, and for ``small'' $T$, $C(x,y)$ is not positive definite on $B(0,1)$. Actually, that threshold of positive definiteness for any spatial dimension $d$ was  determined in 1960 by Fuglede in \cite{Fug}, we recall it in \eqref{TE} below. Surprisingly, this reference 
 does not seem to be cited in the works concerning GMC in the strictly logarithmic case.

We show that for $T\ge T(d)$, where $T(d)$ is the  threshold value, and any ball $B(0,r)$, $r\le 1$, the small deviations are lognormal in any dimension in the sense that 
 \eqref{e:two_sided_estimates} holds for finite positive constants $c_B, C_B$, $C_1 (B), C_2(B)$.
 The lower estimate for $B(0,1)$ is presented in Theorem \ref{mean0}, the upper estimate is given in Theorem \ref{Qor}.
 
 We also study the behavior of the  Laplace transform of GMC related to strictly logarithmic covariance kernel on $B(0,r)$ at infinity for spatial dimension $d\ge 1$ and $r\in(0,1)$. This behavior is actually equivalent to the behavior of small deviations.
 We recall and extend the known results on lower and upper estimates of the form \eqref{e:two_sided_estimates}, as well as for the Laplace transform, showing the  logsquare order of behavior for arbitrary dimension $d$ (Proposition \ref{lowar} and Theorem \ref{kappa}). 
In Theorem \ref{kappa} we  also give a new representation of the Laplace transform of GMC on Euclidean balls in terms of solutions of backward heat equations. In Theorem \ref{INLTB} we also derive an analogous representation for the Laplace transform of the inverse of GMC. 
Thanks to the representation of Theorem \ref{kappa} we are then able to obtain estimates on the optimal constant of the lognormal behavior.
More precisely, we define  
\begin{equation}
 \overline c_r=\sup\{c\ge 0: \exists C>0 \ \textrm{such that } \E e^{-tM(B(0,r))}\le C e^{-c(\ln t)^2} \ \text{for}\ t\ge 1\}.\label{e:cr_1}
\end{equation}
and provide estimates for $\overline c_r$. They are collected in Section \ref{sec:2.2}. In particular, we show that
\begin{equation*}
 \overline{c}_r\approx a(r)=\frac 1{2\gamma^2\ln\frac 1r} \qquad \textrm{as}\ r\to 0
\end{equation*}
 This gives exact behavior of the Laplace transform at infinity for small $r$ since  Proposition \ref{lowar} provides a lower estimate for the Laplace transform. 
We give a number of estimates of $\overline c_1$. These estimates are expressed in terms of $T$ in \eqref{e:covC}, such that $C$ is positive definite on the unit ball $B(0,1)$.

 The article is structured as follows: In Section \ref{sec:main} we state our main results. In Section 3 we provide short preliminaries, while the proofs of results are given in Section 4.

\section{Main results} \label{sec:main}

\subsection{The setting and notation}
\label{sec:notation}
We assume that $Z$ is a generalized Gaussian random field defined on a unit ball $B(0,1)$ in $\R^d$, with covariance function of the form 
\begin{equation}C(x,y) = \ln\frac{T}{|x-y|}
 \label{e:covariance}
\end{equation}
 and where $T$ is a normalizing constant. The constant $T$ is chosen in a way $C(x,y)$ is positive definite on $B(0,1)$. The canonical example for such a setting is the GFF model $(d = 2)$ on the unit disc with $T=1$ (\cite{RV1}).

 The main (and seemingly overlooked by GMC specialists) reference for the problem of positive definiteness of a strictly logarithmic kernel is the seminal work of Fuglede \cite[Eq. (5)]{Fug}
 which states not only that the threshold of positive definiteness of the kernel (\ref{e:covariance}) is finite but determines its explicit value
\begin{align}
\label{TE}
	T(d) = \begin{cases}
    \frac12, & \text{if $d = 1$},\\
    1, & \text{if $d = 2$},\\
		e^{\sum_{k=1}^{d/2 -1}\frac 1{d-2k}}= e^{\frac12 + \ldots + \frac{1}{d-2}}, & \text{if $d > 2$ is even},\\
		\frac 12 e^{\sum_{k=1}^{(d-1)/2 }\frac 1{d-2k}}=\frac12e^{1 + \ldots +\frac{1}{d-2}}, & \text{if $d > 2$ is odd}. 
  \end{cases}
\end{align}
The main result of \cite{Fug} states that for each $T\ge T(d)$  and for  every signed measure $\rho$ on $B$ such that the integral below is well defined we have
\begin{align*}
	\int_{B(0,1)^2}\ln|x-y|\rho(dx)\rho(dy) \leq \rho^2(B(0,1)) \ln T,
\end{align*}
and if $T<T(d)$ the inequality fails for some measure $\rho$. This means that if  $T\ge T(d)$, then \eqref{e:covariance} is positive definite on $B(0,1)$  and  it is not positive definite if $T<T(d)$.
Hence we always assume that $T$ in \eqref{e:covariance} satisfies $T\ge T(d)$.

To determine the value $T(d)$ Fuglede applies the notion of equilibrium measure, say $\xi$, which is a measure for that the logarithmic potential is constant, i.e. if
\begin{align}\label{LogThre}
	\int_B\ln|x-y|\xi(dy) = \ln T, \quad x \in B(0,1).
\end{align}
It turns out that only for $d\in \{1,2\}$ the equilibrium measure coincides with the so called capacitary measure which is a measure for that the logarithmic potential attains minimum value (see \cite[Sect. 2]{Fug}). If $d>2$ the equilibrium measure is not a capacitary measure anymore. 
 
 \medskip

In \cite{RV2} Rhodes and Vargas, by using a Haar measure on the unitary group of $d$-dimensional matrices, give the arguments for the existence of the threshold of positive definiteness not being able to indicate its critical value. 
Actually 
the results of Fuglede \cite{Fug} have been obtained several decades earlier by the very nice concept of equilibrium measures for logarithmic potentials.

\medskip

In what follows, for $\gamma<\sqrt {2d}$ we consider GMC $M_\gamma(\cdot)$ related to $Z$  and defined on $D=B(0,1)$  in the sense of \eqref{e:def_M_gamma}. 
We recall how to understand this formal expression. 
Let $\mu$ be a non-negative measure in $D$ and
\begin{align*}
    \mathcal{M}_{+} = \Big\{\mu: \int_{D\times D}|C(x,y)|\mu(dx)\mu(dy) < \infty\Big\}. 
\end{align*}
Furthermore, let $\mathcal{M}$ be the set of signed measures of the form $\mu = \mu_+ - \mu_-$, where $\mu_+,\mu_- \in \mathcal{M}_+$. Then $Z$ is a 
 field indexed by $\mathcal{M}$ in  the sense that $\int_DZ(x)\mu(dx)$ is a centered Gaussian random variable with variance $\int_{D\times D}C(x,y)\mu(dx)\mu(dy)$.

Following Berestycki \cite{Ber}, one can give a meaning to the formal expression \eqref{e:def_M_gamma} by an approximation in terms of molifiers. (Note that the setup of \cite{Ber} was more general, with more general covariance kernel but with the same type of logarithmic behavior for $x-y$ close to $0$. Also, in our case $\sigma(dx)$ of \cite{Ber} is simply the Lebesgue measure on $\R^d$.)
 
To construct a suitable mollifier we consider a nonnegative Radon measure $\theta$ on $\mathbb{R}^k$ supported in the unit ball such that $\theta(\mathbb{R}^d) = 1$ and
\begin{align}\label{mol}
    \int_{B(0,1)}\big|\ln(|x-y|)\big|\theta(dy) < \infty.
\end{align}
Setting $\theta_{\epsilon}(A) = \theta(A/\epsilon)$, $\epsilon > 0$, $A\in\mathcal{B}(\mathbb{R}^d)$ we obtain a centered, continuous (in a sense of natural field topology) Gaussian field
\begin{align}\label{mollifier}
    Z_{\epsilon}(x) = \int_{\mathbb{R}^d}Z(y)\theta_{\epsilon}(x-y)dy
\end{align}
(see \cite{Ber} for a detailed construction). For $\gamma >0$ and $\epsilon >0$ we may now define a measure
\begin{align} 
\label{GMCmol}
M^{(\epsilon)}_{\gamma}(A) = \int_{A}e^{\gamma Z_{\epsilon}(x) - \frac{\gamma^2}{2}\mathbb{E}Z^2_{\epsilon}(x)}dx, \quad A\in\mathcal{B}(\mathbb{R}^d),
\end{align}
and it turns that under the condition $\gamma<\sqrt{2d}$ the family $(M^{(\epsilon)}_{\gamma}(A))_{\epsilon}$ converges with $\epsilon\rightarrow 0$ to some $M_{\gamma}(A)$  in probability and $L^1(\mathbb{P})$. Actually, $(M^{(\epsilon)}_{\gamma}(\cdot))_{\epsilon}$ as measures converge in probability towards measure $M_{\gamma}(\cdot)$ called GMC, for the topology of weak convergence of measures on the domain $D$ \cite[Thm. 1.1]{Ber}.

\medskip

If the covariance kernel is given by (\ref{e:covariance}), $M_{\gamma}(\cdot)$ has a scaling property 
which means that it satisfies the following equality in law
\begin{align}\label{scal}
	M_{\gamma}(B(0,r)) \stackrel{law}{=} r^de^{\gamma\Omega_r - \frac{\gamma^2}{2}\mathbb{E}\Omega^2_r}M_{\gamma}(B(0,1)), \quad r\in (0,1),
\end{align}
where $\Omega_r$ is a centered Gaussian random variable with variance $\ln\frac{1}{r}$ independent of $M_{\gamma}(\cdot)$. Informally, this is the consequence of the fact that for $0<r<1$ the generalized random field $(Z(rx))_{x\in B(0,1)}$ has the same law as $(Z(x)+\Omega_r)_{x\in B(0,1)}$ with $\Omega_r$ as above, independent of $Z$. This has been formalized in \cite{RV1}.

We fix $0<\gamma<\sqrt{2d}$ and we denote  $M_1 = M_{\gamma}(B(0,1))$. Accordingly, by $M_r$ we denote the GMC on $B(0,r)$, $r\in (0,1)$. Note that we drop $\gamma$ in this notation, but the parameter $\gamma$  
is present in the model and various constants will depend on $\gamma$. We  will often denote $B = B(0,1)$.

\medskip

Later on we will  work out a description of small deviations of $M_r$  related to strictly logarithmic covariance kernel in terms of Laplace transforms. 
The corresponding notation is given at the beginning of Sections \ref{sec:Laplace} and \ref{sec:2.2} (see in particular \eqref{e:Qr_def}, \eqref{e:ar}, \eqref{e:cr} for the definitions of $Q_r(x)$, $a(r)$, $b$ and $\overline{c}_r$) .

\subsection{Lognormality on $B(0,1)$}
\label{sec:2.3}

We will show that the small deviations of GMC on $B(0,1)$ with strictly logarithmic kernel are in all dimensions always lognormal. As it has been already said in Introduction, there were some earlier results to that effect, either for particular dimensions or with additional assumptions which are not straightforward to check (as in the method of eigenfunctions), or in a somewhat different setting. On the other hand, recent results of Lacoin, Rhodes and Vargas \cite{Lac} show that rate of decay of GMC in the vicinity of $0$ in case of vanishing spatial mean, i.e. when $\int_{D}Z(x)dx = 0,$ is much faster than lognormal, i.e.
$$P(M_1 \leq \delta) \leq C e^{c \delta^{-4/\gamma^2} |\ln \delta|^k},$$ for positive constants $c,C, k$ \cite[p. 490]{Lac}.
Although the last result is presented for $d=2$, the authors claim their methodology works also in any higher dimension. 
However, as we will see, this situation does not arise in our setting. 

\begin{theorem} \label{mean0}
 Assume \eqref{e:covariance} with $T\ge T(d)$ given by \eqref{TE}, so that \eqref{e:covariance} is positive definite on $B(0,1)\subset \R^d$.  Then there always  exist finite positive constants $c, C$ such that
\begin{align}\label{logrowth}
	\mathbb{P}(M_1 \leq \delta) \geq Ce^{- c(\ln\delta)^2}, \quad \textit{for any }\ 0<\delta < 1.
\end{align} 
If $d\geq 3$  then in \eqref{logrowth} one can take any $c$ satisfying
\begin{align} \label{ce}
	c >  \frac {\ln T - \frac{1}{|B(0,1)|^2}\int_{B^2(0,1)}\ln|x-y|dxdy}{2 \gamma^2 (\ln T-\ln T^*(d))^2},
\end{align}
where 
\begin{align}\label{Tstar}
	T^*(d) = \exp\Big\{\frac{1}{|B(0,1)|}\int_{B(0,1)}\ln {|e-y|}dy\Big\},
	\end{align}
	and $e\in \partial B(0,1)$.
\end{theorem}

\begin{remark}\label{Rem}
Let us notice that in the case $T>T(d)$ the arguments for lognormality can be deduced from the fact that in this case the random field $Z$ can be extended to a larger ball and then applying the scaling property \eqref{scal} (see the beginning of the proof of Theorem \ref{mean0}). The real difficulty lies at the threshold of positive definiteness, $T=T(d)$. Our main argument, which works for dimensions $d\ge 3$ relies on the use of Lemma \ref{mean} in Section \ref{sec:proofs}, since in this case we have that $T^*(d)<T(d)$, where $T^*(d)$ is given by (\ref{Tstar}).
	Clearly, in this case we also have 
	\begin{equation*}
	 \exp\Big\{\frac{1}{|B(0,1)|^2}\int_{B(0,1)^2}\ln {|x-y|}dxdy\Big\}\le T^*(d)< T(d)
	\end{equation*}
Hence $\Var (\int_{B(0,1)}Z(x)dx)>0$. By direct computation, this variance is also non-zero for $d=1$ and $d=2$.
\end{remark}

\medskip

For the sake of completeness  we will give a proof of the upper estimate of the small deviations of GMC on $B(0,1)$ and the logsquared order of convergence of the Laplace transform.   
The elegant proof of this fact using Markov property for GFF ($d=2$) can be found in \cite[Sect. 4.1]{Aru}. We will use Kahane's convexity inequality in the case of arbitrary natural $d$.

\begin{theorem}\label{Qor} 
Assume \eqref{e:covariance} with $T\ge T(d)$. Then
(i) There exist positive constants $c,C$, for which
\begin{align}
	\mathbb{E}e^{-tM_1} \leq Ce^{- c(\ln t)^2}, \quad t > 1.
	\label{e:8a}
\end{align}
(ii) There exist positive constants $c_1, C_1$ such that
for every $0<\epsilon<1$
\begin{align}
	\mathbb{P}(M_1 \leq \epsilon) \leq C_1e^{-c_1(\ln\epsilon)^2}.
	\label{e:9}
\end{align}
\end{theorem}

 \subsection{The Laplace transform of GMC on a ball $B(0,r)$, $r<1$.}
 \label{sec:Laplace}

In what follows we will work out a description of small deviations of GMC on a ball $B(0,r)$ with $r<1$, related to strictly logarithmic covariance kernel in terms of Laplace transforms. For $r\in (0,1)$ we set
\begin{align}
	a(r) := \frac{1}{2\gamma^2\ln\frac{1}{r}}, \quad b := \frac12 + \frac{d}{\gamma^2}, \quad C(r) := \frac{r^{\frac{\gamma^2}{8} + \frac{d}{2} + \frac{d^2}{2\gamma^2}}}{\sqrt{2\pi \gamma^2 \ln\frac{1}{r}}}.
\label{e:ar}
	\end{align}
and 
\begin{align}
	Q_r(x) := \mathbb{E}e^{-e^xM_{\gamma}(B(0,r))}, \quad r\in (0,1], \ x \in\R.
	\label{e:Qr_def}
\end{align}
We proceed with the following lower bound on $Q_r(x)$. 
\begin{proposition}\label{lowar} Let $r\in(0,1)$. Then for every $x > \frac{b}{2a(r)}$
\begin{align}
	Q_r(x) \geq C(r)\widehat{c}(r)e^{-a(r)x^2 + \frac{b^2}{a(r)}},
\end{align}
where $\widehat{c}(r)$ is given by
\begin{align*}
\widehat{c}(r) = \mathbb{P}\Big(|\ln M_1|\leq \frac{b}{ 4a(r)}\Big)\Big(e^{-e^{-\frac{b}{ 4a(r)}}} -e^{-e^{\frac{b}{ 4a(r)}}}\Big).
\end{align*}
\end{proposition}
To study the behavior of
 the Laplace transform of GMC at infinity,   we work out a new representation of $Q_r(x)$. Let us denote
\begin{align}\label{kap}
	\kappa(s,x) = \mathbb{E}\Big(\frac{1}{M_1}e^{(\frac{x}{2}- b - 1)\ln\frac{T_1}{M_1}- \frac{s}{4}\big(\ln\frac{T_1}{M_1}\big)^2}\Big), \quad s>0, x > 2b +2. 
\end{align}
where $T_1$ is and exponential random variable with parameter $1$ independent of $M_1$. 
 We have the following representation of $Q_r$:

\begin{theorem}\label{kappa} Let $r \in (0,1)$ and let $\kappa$ be given by
 \eqref{e:kappa}. 
Then
\begin{align}\label{e:Qr}
	Q_r(x) = C(r) e^{-a(r)x^2 + bx}\kappa(4a(r), 4a(r)x), \quad x > \frac{b}{2a(r)}.
\end{align}
Moreover, the function $\kappa(\cdot,\cdot)$  is a solution of the backward heat equation
\begin{align}\label{bckwd}
	\frac{\partial\kappa(s,z)}{\partial s} = - \frac{\partial^2\kappa(s,z)}{\partial z^2}, \quad s >0, z > 2b
\end{align}
in class of non-negative functions such that
\begin{enumerate}
\item $\kappa$ is convex, decreasing to $0$ in variable $s$,
\item $\kappa$ is convex in variable $z$,
\item $\frac{\partial\kappa}{\partial s} + \frac{\partial\kappa}{\partial z} \leq \frac14\kappa$ for $s >0$, $z > 2b$,
\item If $c_0$ is such that 
\begin{equation}Q_1(x)\le C_0e^{-c_0x^2}\qquad \text{for}\  x\ge t_0,
\label{e:c0}
\end{equation}
 then for any $p>1$
there exists a constant $L_p>1$ such that for every $s\ge 0$ and $x\ge 2b+2$
\begin{equation}
	\kappa(s,x) \leq
	 L_p^x
	e^{\frac{x^2p}{4(ps + 4c_0)}}.
	\label{e:kappa_est1}
\end{equation}
\end{enumerate}
Furthermore, there exists $C_1(p,r)>0$ and $x_0(r)>0$ such that for 
$c_0$ as in point 4, $x>x_0(r)$
 we have
\begin{align}\label{oda1}
	-a(r)x^2 -C_1(p,r)x\leq \ln Q_r(x) \leq -\frac{a(r)c_0}{pa(r) + c_0}x^2+C_1(p,r)x.
\end{align}
\end{theorem}
\medskip

\ \\
Surprisingly the Laplace transform of the inverse of GMC on a ball can be represented in a similar way.
We will use the notation (\ref{e:ar}). Let $T_1$ be exponential with parameter $1$ independent of $M_1$. Define
\begin{align}\label{kappa2}
\kappa_2(s,x) := \mathbb{E}M_1e^{(\frac{x}{2}+ b - 1)\ln(M_1 T_1) - \frac{s}{4}(\ln(M_1 T_1))^2}, \quad x >0, \ s>0.
\end{align}
We have
\begin{theorem}\label{INLTB}	For any $r\in (0,1)$ and $x>0$
\begin{align}\label{ILTB}
	\mathbb{E}e^{-\frac{e^x}{M_r}} = C(r)e^{-a(r)x^2 - bx}\kappa_2(4a(r), 4a(r)x),
\end{align}
where $\kappa_2(\cdot,\cdot)$ is defined by (\ref{kappa2}). Moreover, (\ref{kappa2}) solves the Cauchy problem 
\begin{align}\label{2PDE}
	\frac{\partial \kappa_2}{\partial s} + \frac{\partial^2 \kappa_2}{\partial x^2} = 0, \quad (s,x) \in (0,\infty)\times (0,2(q-b)),
\end{align}
with initial condition $\kappa_2(0,x) = \Gamma(\frac{x}{2} + b)\mathbb{E}M_1^{\frac{x}{2} + b}$, in the class of bounded functions, convex with respect to the both arguments $s$ and $x$, where $q = \frac{2d}{\gamma^2}$.  Moreover, for large values of $x$
\begin{align}\label{lnka}
	\ln\kappa_2(4a(r), 4a(r)x) = a(r)x^2 + (b-q)x + O(x^{\delta}), \quad \delta \in (0,1).
\end{align}
\end{theorem}

\subsection{Estimations of constants in small deviations}
\label{sec:2.2}
In what follows we assume that $T\geq T(d)$, where the last is given by (\ref{TE}). For $0<r\le 1$ define
\begin{equation}
 \overline c_r=\sup\{c\ge 0: \exists C>0 \ \textrm{such that } \E e^{-tM_r}\le C e^{-c(\ln t)^2} \ \text{for}\ t\ge 1\}.\label{e:cr}
\end{equation}
Observe, that $\overline c_1<\infty$ by Theorem \ref{mean0}. 
As a consequence  of \eqref{oda1}, the fact that $\lim_{r\to0}a(r)=0$ and Remark \ref{Rem}, we obtain:
\begin{theorem}\label{thm:cr_bound} 
For any $0<r<1$ we have
\begin{equation}
 \label{e:cr_estim1}
 \frac{a(r)\overline c_1}{a(r)+\overline c_1} \le \overline c_r\le a(r), \quad \lim_{r\to 0}\frac {\overline{c}_r}{a(r)}=1.
\end{equation}
Moreover, for $d\geq 3$, 
\begin{align}\label{uci}
	\overline c_1 \leq \frac{\int_{B^2(0,1)}\ln\frac{T}{|w-v|}dvdw}{2\gamma^2(\int_{B(0,1)}\ln\frac{T}{|e-v|}dv)^2} = \frac {Var(\Omega)}{2 \gamma^2 |B(0,1)|^2\Big(\ln\frac{T}{T^*(d)}\Big)^2},
\end{align}
where $e\in \partial B(0,1)$, $\Omega = \int_{B(0,1)}Z(x)dx$ and $T^*(d)$ is defined by (\ref{Tstar}). 
\end{theorem}

The lower bounds on optimal constants of small deviations are collected in the following theorem.  
\begin{theorem}\label{cezero} If $C(x,y) = \ln\frac{T}{|x-y|}$ then: 
\begin{itemize}
 \item[(i)] If $T\le 2$
then
\begin{align}\label{cees3}
\overline c_1 \geq a\left(\frac 12-\frac T4\right).
\end{align}
In particular, if $T\le 1$ then
\begin{align}\label{cees}
\overline c_1 \geq a\left(\frac 14\right).
\end{align}
\item[(ii)] If $T>1$ then for every $p\in (0,1 - 1/\sqrt{2})$
\begin{align}\label{cees33}
\overline c_1 \geq (2(1-p)^2 -1)a(4^{-1})\wedge \frac{p^2}{2\gamma^2\ln T}.
\end{align}
\item[(iii)] If $T\leq 1$
then 
\begin{align}\label{cees2}
	\overline c_1 \geq (2d - 1)a\Big(\frac {3-2\sqrt 2}{2}\Big). 
\end{align}
\item[(iv)] If $T>1$ then for every $p\in (0,1 - 1/\sqrt{2d})$
\begin{align}
	\overline c_1 \geq (2d(1-p)^2 - 1)a\Big(\frac {3-2\sqrt 2}{2}\Big) \wedge \frac{p^2}{2\gamma^2\ln T}.
\end{align}
\end{itemize} 
\end{theorem}

Actually, we can say a little bit more about the relation between $\overline c_r$ and $\overline c_1$. 
\begin{theorem} \label{CNQR} 
If $q,r \in (0,1)$ satisfy 
\begin{align}\label{conqr}
	\overline c_r -\frac{a(r)}{2}(q^2 + 1) \geq  0,
\end{align}
then
\begin{align}
	\overline c_1 \geq \frac{4\overline c_r - 2a(r)(q^2 + 1)}{(1-q)^2}.
\end{align}  
\end{theorem}

Observe, that condition \eqref{conqr} is always non-empty. In fact, from Theorem \ref{thm:cr_bound} it follows that for any $q\in(0,1)$  \eqref{conqr} is satisfied at least for any $r$ sufficiently small.

\section{Proofs and further discussion}
 
\label{sec:proofs}

Recall that we assume that the underlying Gaussian random field has covariance given by \eqref{e:covariance}. And the corresponding measure $\sigma$ in the definition of $M_\gamma$ is the Lebesgue measure.
For $r\in(0,1]$ let us denote 
\begin{equation}
M_r = M_{\gamma}(B(0,r)).
\label{e:Mr}
\end{equation} 

To prove Theorem \ref{mean0} we will use the following lemma. 

\begin{lemma}\label{mean} Assume that $C(x,y) = \ln\frac{T}{|x-y|}$ is positive definite on $B(0,1)\subset \R^d$. Denote 
\begin{equation*}
 \Omega=\int_{B(0,1)}Z(x)dx
\end{equation*} 
and 
\begin{equation*}
\ln T^*(d)=\frac 1{|B(0,1)|}\int_{B(0,1)}\ln {|e-v|}dv 
\end{equation*}
for  $e\in \partial B(0,1)$.

If 
\begin{equation}
 \label{e:cond_T}
 \ln T> \ln T^*(d) 
\end{equation}
then $\Var (\Omega)=\int_{B(0,1)^2} \frac {\ln T}{|x-y|}dxdy>0$ and  
there exist  positive constants $c, C$ such that such that (\ref{logrowth}) is true.
More precisely, there exists $C$ such that this is true for any 
\begin{equation}
c> \frac {\Var (\Omega)}{2 \gamma^2 |B(0,1)|^2(\ln T-\ln T^*(d))^2}.
\end{equation}
 In particular, if $T\geq 2$ the lognormality is universal.
\end{lemma}
\begin{proof} (of Lemma \ref{mean})\\
 Recall that
\begin{equation*}
 \Omega=\int_{B(0,1)}Z(x)dx.
\end{equation*}
We will show that if for 
 $e\in \partial B(0,1)$
 \eqref{e:cond_T} holds, i.e.
\begin{align*}
 \ln T> \frac 1{|B(0,1)|}\int_{B(0,1)}\ln {|e-v|}dv, 
\end{align*}
then $\Var (\Omega)>0$ and  
there exist  positive constants $c, C$ such that (\ref{logrowth}) holds. 
We  have $\Omega = \int_{B(0,1)}Z(x)dx$ which is centered Gaussian random variable of finite variance of the form
\begin{align*}
Var(\Omega) = \int_{B^2(0,1)}\ln\frac{T}{|x-y|}dxdy  .
\end{align*}
Observe that the function $x\mapsto \int_{B(0,1)}\ln\frac{T}{|x-y|}dy$ is a decreasing function of $|x|$, hence condition \eqref{e:cond_T} implies that $Var (\Omega)>0$.

We aim at an orthogonal decomposition in sense of independence of $(Z(x) - \alpha(x)\Omega)_{x\in B(0,1)}$ and $\Omega$ for some appropriately specified function $\alpha(\cdot)$. It is easy to verify that this is true for $\alpha(x) = \frac 1{\mathbb{E}\Omega^2}\int_{B(0,1)}\ln\frac{T}{|x-v|}dv$. We observe that $\alpha(\cdot)$ is bounded from above and also that
\begin{equation}
 \label{e:infalpha}
\underline  \alpha:= \inf_{x\in B(0,1)}\alpha(x)=\alpha(e) = \frac 1{\mathbb{E}\Omega^2}\int_{B(0,1)}\ln \frac T{|e-v|}dv>0.
\end{equation}
where $\Vert e\Vert =1$. The last inequality follows again from \eqref{e:cond_T}.
 
Let us denote
 \begin{equation}\label{dcmp} 
\widetilde{Z}(x) = Z(x) - \alpha(x)\Omega, \qquad x\in B(0,1),
 \end{equation}
 and observe that it is a centered logarithmic Gaussian field so it is reasonable to consider a GMC built on $\widetilde{Z}(\cdot)$ and denoted by $\widetilde{M}$. Choose $\delta > 0$ and write
\begin{align*}
	\mathbb{P}(M_1 \leq \delta) = \mathbb{P}\Big(\int_{B(0,1)}e^{\gamma\alpha(x)\Omega - \frac12\gamma^2\alpha^2(x)\mathbb{E}\Omega^2}\widetilde{M}(dx)\leq \delta \Big).
\end{align*}
We have
\begin{align}\label{c1u}
	\mathbb{P}\Big(\int_{B(0,1)}e^{\gamma\alpha(x)\Omega - \frac12\gamma^2\alpha^2(x)\mathbb{E}\Omega^2}\widetilde{M}(dx)\leq \delta \Big)
\ge & \P\left(\Omega<0, e^{\underline \alpha \gamma\Omega-\frac 12\underline \alpha^2\gamma^2 \E \Omega^2}\widetilde M(B(0,1))\le \delta\right)\notag\\
\ge & \P(\Omega<0, e^{\underline{ \alpha}\gamma\Omega}\le \delta)P(\widetilde M(B(0,1))\le 1)\notag\\
\ge &  P(\xi\ge \frac {-\ln \delta}{\underline \alpha \gamma \sqrt{\mathbb{E}\Omega^2}})P(\widetilde M(B(0,1))\le 1),
\end{align}
where $\xi$ has standard normal distribution.
Clearly the first factor is of lognormal type, hence the first conclusion of the Lemma follows. The second point follows immediately since under the assumption of positive definiteness and $T\geq 2$  condition \eqref{e:cond_T} is always true. 
\end{proof}

\begin{proof} (of Theorem \ref{mean0})\\
Recall that $T(d)$ is given by (\ref{TE}). Let us observe that the case $T>T(d)$ is easy. Indeed, in this case the GMC may be defined on a larger ball $B(0,R)$ with $R>1$. 
Analogously, as in \eqref{scal} we can write
\begin{align}\label{scal2}
	M_1\stackrel{law}{=} R^{-d}e^{\gamma\xi - \frac{\gamma^2}{2}\mathbb{E}\xi^2}M_R,
\end{align}
where $\xi$ is a nontrivial centered Gaussian random variable independent of $M_R$. This implies the lognormal type of small deviations for $M_1$.

The non-trivial case is $T=T(d)$, since then we cannot extend GMC to a larger ball.

\medskip
Let $T^*(d)$  be the boundary value from criterion of Lemma (\ref{mean}), i.e. 
\begin{align*}
	\ln T^*(d) = \frac 1{|B(0,1)|}\int_{B(0,1)}\ln {|e-y|}dy.
\end{align*}
Let us also observe, that $T^*(d)\le 2$.

By Lemma \ref{mean}, to prove lognormality it suffices to show that $T^*(d)< T(d)$. Recalling the explicit form of $T(d)$, this is absolutely clear for $d\ge 6$ since then we have $T^*(d)\le 2<T(d)$. We will show that the inequality $T^*(d)<T(d)$ also holds for $d\ge 3$.

\ \\
Before concluding the remaining cases,  let us determine the formula for $\ln T^*(d) = \ln T^*(d)$. We will write $B(0,1) = B_d(0,1)$. By the definition of $T^*(d)$, without loss of generality, we can choose $e = (1,0,\ldots,0)$ and write
\begin{align*}
	\ln T^*(d) = \frac{1}{|B_d(0,1)|}\int_{B_d(0,1)}\ln|e-x|dx = \frac{1}{|B_d(0,1)|}\int_{B_d(e,1)}\ln|x|dx.
\end{align*}
We compute the last integral using polar coordinates. We obtain
\begin{align*}
	\ln T^*(d) = \frac{1}{|B_d(0,1)|}\int_0^2 (\ln r) A_d(\phi(r))dr,
\end{align*}
where $A_d(\phi(r))$ denotes the area of a hyperspherical cap in a $d$ sphere of radius $r$, where $\phi(r) = \arccos(r/2)$ denotes  the colatitude angle, i.e. the angle between the vector of a sphere and its $d$-th positive axe. The formula for $A_d(\phi(r))$ is known
\begin{align*}
	A_d(\phi(r)) = \frac{2\pi^{\frac{d-1}{2}}}{\Gamma(\frac{d-1}{2})}r^{d-1}\int_0^{\phi(r)}\sin^{d-2}(\theta)d\theta, \quad \phi(r) = \arccos(r/2),
\end{align*}
(see \cite{Li}).

\ \\
\emph{Case $d = 1$}\\
In this case the small deviations of GMC are known to be lognormal by the result of \cite{Remy}. Notice however, that the criterion of Lemma  \ref{mean} does not hold, since
\begin{align*}
	\ln T^*(1) = \frac12\int_0^2\ln y dy = \ln 2 - 1 > \ln T(1) = \ln\frac12.
\end{align*}

\ \\
\emph{Case $d = 2$}\\
By the computation in WolframAlpha we find that in this case $T^*(2) = T(2) = 1$, so the criterion of Lemma  (\ref{mean}) does not apply. However, the case $d=2$ coincides with the well developed Liouville theory. In particular, for the Dirichlet's  Laplacian eigenvalue problem in $B(0,1)\subset \mathbb{R}^2$ it is well known that the first eigenfunction is positive so that the assumptions of \cite[Ex. 7(c), Ch. 3]{Ber1} are satisfied and the lognormality is claimed to follow by applying  Karhuhen–Loeve expansion.

\ \\
\emph{Case $d = 3$}\\
In case $d= 3$ the area of a hyperspherical cap is equal to $A_3(\phi(r)) = 2\pi (r^2 - \frac{r^3}{2})$, so that
\begin{align*}
	\ln T^*(3) = \frac{1}{\frac{4}{3}\pi}\int_0^2(\ln r) 2\pi (r^2 - \frac{r^3}{2})dr = \ln 2 - \frac{7}{12} \approx 0.1098,
\end{align*}
while
\begin{align*}
	\ln T(3) = \ln \frac{e}{2} = 1 - \ln 2 \approx 0.3068,
\end{align*}
so that lognormality follows by  criterion of Lemma  \ref{mean}. 

\ \\
\emph{Case $d = 4$}\\
If $d=4$ and $\phi(r) = \arccos(r/2)$ then
\begin{align*}
	A_4(\phi(r)) &= \frac{2\pi^{\frac{3}{2}}}{\Gamma(\frac{3}{2})}r^{3}\int_0^{\phi(r)}\sin^{2}(\theta)d\theta = \pi r^3(2\phi(r) - \sin(2\phi(r)))\\
	&= 2\pi r^3(\arccos(r/2) - \frac{r}{2}\sqrt{1 - r^2/4}).
\end{align*}
Computing $\ln T^*(4)$ using integral calculator \cite{I} gives $\ln T^*(4) \approx 0.1666$, while $\ln T(4) = \frac12$. The assertion follows by  criterion of Lemma  \ref{mean}.

\ \\
\emph{Case $d = 5$}\\
If $d = 5$ we have
\begin{align*}
	\int_0^{\phi(r)}\sin^3\theta d\theta &= \int_{\cos(\phi(r))}^{1}(1-y^2)dy = 1 - \cos(\phi(r)) - \frac{1 - \cos^3(\phi(r))}{3}\\
	&= \frac23 - \frac{r}{2} + \frac{r^3}{24}. 
\end{align*}
For this reason
\begin{align*}
\ln T^*(5) = \frac{15}{8\pi^2}\int_0^2(\ln r) 2\pi^2\Big(\frac23 - \frac{r}{2} + \frac{r^3}{24}\Big)dr = \ln 2 - \frac{59}{120}\approx 0.2014,
\end{align*}
while
\begin{align*}
	\ln T(5) = \ln \frac{e^{1+\frac{1}{3}}}{2} = \frac13 + 1 - \ln 2 \approx 0.6402,
\end{align*}
so that lognormality again follows by  criterion of Lemma  \ref{mean}. 
\end{proof}

\begin{remark}
Let us denote
\begin{align}\label{zeroVar}
\ln T_0=\frac 1{|B|^2}	\int_{B^2}\ln|x-y|dxdy. 
\end{align}
 The condition for vanishing spatial mean is actually equivalent to the condition\\
 $Var (\int_BZ(x)dx) = 0$, that is $T_0=T$. By positive definiteness of $C$ we always have $T_0\le T(d)$. As we have seen in the proof of Theorem \ref{mean0}, for $d\ge 3$ we have $T_0\le T^*(d)<T(d)$. Hence the spatial mean does not vanish, and 
 the lognormality of small deviations of GMC cannot  spoiled by the arguments of the case of vanishing spatial mean due to Lacoin, Rhodes and Vargas \cite{Lac}.
 For $d=1$ a direct computation shows that
\begin{align*}
	\ln T_0(1) = \frac14\int_{-1}^1\int_{-1}^1\ln|x-y|dxdy = \ln2 -2 < \ln T(1),
\end{align*}
and for $d=2$ we have $T_0<T^*(d)=T(d)$, so the spatial mean also does not vanish. 
\end{remark}

In what follows we prove the results giving the description of small deviations of GMC in terms of Laplace transforms. The starting point of recognizing the structure of the Laplace transform of $M_r$
 is the following technical lemma. 
\begin{lemma}\label{L} 
Assume that $T_1$ is an exponential random variable with parameter $1$  independent of $M_1$ then for any $r\in(0,1)$ and $x\in B(0,1)$ such that $B(x,r)\subset B(0,1)$ and for $t>1$ we have
\begin{align}
\mathbb{E}e^{-tM_{\gamma}(B(x,r))} = \frac{r^{\frac {\gamma^2}8+\frac d2+\frac {d^2}{2\gamma^2}}t^{\frac 12+\frac d{\gamma^2}}}{\sqrt{2\pi \gamma^2  \ln\frac1r}}
\mathbb{E}\Big( T_1^{-\frac 32-\frac d{\gamma^2}}M_1^{\frac 12+\frac d{\gamma^2}}e^{-\frac{(\ln T_1-\ln t - \ln M_1)^2}{2\gamma^2\ln\frac 1r}}\Big).
\label{e:LaplMr}
\end{align}
\end{lemma}
\begin{proof} 
Let $t>1$. Clearly, by the form of the covariance $C$ and since $\sigma(dy)=dy$  $M_\gamma(B(x,r))$ has the same law as $M_r$. It is therefore enough to consider $M_r$.

Let us denote by $J$ the right hand side of \eqref{e:LaplMr}.
By the scaling property \eqref{scal} we have
\begin{align*}
J:=	\mathbb{E}e^{-tM_r} = \mathbb{E}e^{-tr^de^{\gamma\Omega_r - \frac{\gamma^2}{2}\mathbb{E}\Omega^2_r}M_1},
\end{align*}
where $\Omega_r$ is a centered Gaussian random variable with variance $\ln \frac 1r$ independent of $M_1$.

By independence we have
\begin{equation*}
 J=\frac1{\sqrt{2\pi \ln\frac 1r}}\int_\R \mathbb{E}e^{-tr^de^{\gamma y- \frac{\gamma^2}{2} \ln \frac 1r }M_1}e^{-\frac {y^2}{2\ln \frac 1r}}dy.
\end{equation*}
Substituting $s = r^de^{\gamma y- \frac{\gamma^2}{2} \ln \frac 1r }$ we obtain 
\begin{align}
\label{e:7}
	J &= \frac{1}{\sqrt{2\pi \gamma^2  \ln\frac1r}}
	\mathbb{E}\int_0^{\infty}e^{-tM_1s}s^{-1}e^{-\frac{\left(\ln s+(\frac {\gamma^2}2+d)\ln\frac 1r\right )^2}{2\gamma^2\ln\frac1r}}ds.
	\end{align}
	We write
	\begin{align*}
	 \exp\left\{-\frac{\left(\ln s+(\frac {\gamma^2}2+d)\ln\frac 1r\right )^2}{2\gamma^2\ln\frac1r}
	 \right\}
	 =&\exp\left\{-\frac{(\ln s)^2}{2\gamma^2\ln \frac 1r}
	 -\frac {(\frac {\gamma^2}2+d)\ln s}{\gamma^2}
	 -\frac {(\frac {\gamma^2}2+d)^2\ln 
 	 \frac 1r}{2\gamma^2}
\right\}\\
=&\exp\left\{-\frac{(\ln s)^2}{2\gamma^2\ln \frac 1r}\right\} s^{-\frac 12-\frac d{\gamma^2}}r^{(\frac {\gamma^2}2+d)^2/(2\gamma^2)}
	\end{align*}
Plugging this into \eqref{e:7} we obtain
	\begin{align}
J	&= \frac{1}{\sqrt{2\pi \gamma^2  \ln\frac1r}}
r^{\frac {\gamma^2}8+\frac d2+\frac {d^2}{2\gamma^2}}
J_1,
\label{e:8}
\end{align}
where
\begin{align*}
J_1=&
\mathbb{E}\int_0^{\infty}e^{-tM_1s}s^{-\frac 32-\frac d{\gamma^2}} e^{-\frac{(\ln s)^2}{2\gamma^2\ln\frac1r}}ds. 
\end{align*}
Substituting $u=stM_1$ and then introducing a random variable $T_1$ independent of $M_1$ and such that $T_1$ has exponential distribution with parameter $1$ we obtain
\begin{align*}
J_1=&\mathbb{E}\int_0^{\infty}e^{-u}u^{-\frac 32-\frac d{\gamma^2}} (tM_1)^{\frac 12+\frac d{\gamma^2}} e^{-\frac{(\ln u-\ln t-\ln M_1)^2}{2\gamma^2\ln\frac1r}}ds\\
=& t^{\frac 12+\frac d{\gamma^2}}\mathbb{E} T_1^{-\frac 32-\frac d{\gamma^2}}M_1^{\frac 12+\frac d{\gamma^2}}e^{-\frac{(\ln T_1-\ln t - \ln M_1)^2}{2\gamma^2\ln\frac 1r}}.
\end{align*}
\end{proof}

It is well known that the behavior of the Laplace transform at infinity corresponds to the behavior of the distribution function at $0$. In our setting we will need the following lemma.

\begin{lemma}\label{lem:lapl_distrib}
 Let  $\xi$ be a nonnegative random
 variable. \\
 (a) Assume that there exist positive constants $C_1,c, \varepsilon_0$ such that for any $0\le \varepsilon\le \varepsilon_0$ we have
 \begin{equation}
  \label{e:distribution_f}
P(\xi\le \varepsilon)\le C_1 e^{-c(\ln \varepsilon)^2}.
  \end{equation}
Then, for any $0<\delta<1$, there exists $t_0$ such that for any $t\ge t_0$ we have
\begin{equation}
 \E e^{-t\xi}\le (C_1+\delta)e^{-c(1-\delta)(\ln t)^2}
 \label{e:laplace}
\end{equation}
More precisely, suppose that $\kappa>0$ and  $t_0> e$ is such that for any $t\ge t_0$ we have $(\ln t)^{2+\kappa}\le \varepsilon_0 t$, then
\begin{equation}
 \E e^{-t\xi}\le C_1 e^{-c(\ln t-(2+\kappa)\ln\ln t)^2}+e^{-(\ln t)^{2+\kappa}} \qquad \textrm{for any}\ t\ge t_0.
\label{e:kappa}
 \end{equation}
(b) If there exist  positive constants $C_1,c, \varepsilon_0$ such that for any $0\le \varepsilon\le \varepsilon_0$ we have
 \begin{equation}
  \label{e:distribution_f_lower}
P(\xi\le \varepsilon)\ge  C_1 e^{-c(\ln \varepsilon)^2}.
  \end{equation}
  then for $t\ge \left(1\vee \frac 1\varepsilon_0\right)$ we have
  \begin{equation}
   \E e^{-t\xi}\ge \frac {C_1}2 e^{-1}e^{-c(\ln t+ \ln 2)^2}.
   \label{e:laplace_lower}
  \end{equation}
\end{lemma}

\proof Part (a).
We have
\begin{equation}
 \E e^{-t\xi}=\int_0^\infty \P(e^{-t\xi}\ge u)du
 =\int_0^1 \P(e^{-t\xi}\ge u)du=
 \int_0^\infty P(\xi\le \frac rt)e^{-r}dr,
\label{e:laplace_1}
 \end{equation}
where in the last equality we substituted $r=-\ln u$. We write
\begin{equation*}
 \E e^{-t\xi}=I(t)+II(t),
\end{equation*}
where 
\begin{align*}I(t)
=&\int_0^{(\ln t)^{2+\kappa}} \P(\xi\le \frac rt)e^{-r}dr,\\
II(t)=&\int_{(\ln t)^{2+\kappa}}^\infty \P(\xi\le \frac rt)e^{-r} dr.
\end{align*}
Clearly $II(t)\le e^{-(\ln t)^{2+\kappa}}$, thus giving the second term on the right hand side of \eqref{e:kappa}.
If $t$ is sufficiently large so that 
$\frac {(\ln t)^{2+\kappa}}{t}<\varepsilon_0$, to estimate $I(t)$ 
we use  assumption \eqref{e:distribution_f}:
\begin{equation*}
 I(t)\le C_1 \int_0^{(\ln t)^{2+\kappa}}e^{-c(\ln{\frac rt})^2}e^{-r}dr.
\end{equation*}
Note that on the set of integration $(\ln \frac r t)^2=(\ln t-\ln r)^2\ge (\ln t -(2+\kappa)\ln\ln t)^2$, hence \eqref{e:kappa} follows.

\ \\
(b) Using \eqref{e:laplace_1} for $t\ge \frac 1{\varepsilon_0}$ we have
\begin{equation*}
 \E e^{-t\xi}\ge \int_{\frac 12 }^1\P(\xi\le \frac rt)e^{-r}dr\ge 
 C_1e^{-1}\int_{\frac 12}^1 e^{-c(\ln t-\ln r)^2}dr
\end{equation*}
and \eqref{e:laplace_lower} follows.
\qed

\begin{proof} (of Theorem \ref{Qor})    
Let $\delta, r\in (0,1)$ be such that $\delta<T$, $\delta+4r\le 2$ and choose $x_1,x_2\in B(0,1)$ such that the balls $B(x_1,r)$ and $B(x_2,r)$ are both contained in $B(0,1)$ and their distance is at least $\delta$.
Clearly
\begin{align}
	\mathbb{E}e^{-tM_1} \leq \mathbb{E}e^{-tM(B(x_1,r)) - tM(B(x_2,r))}.
	\label{e:a}
\end{align}
If GMC is built on the Gaussian field $(Z(x))$ then we consider field $$((Z(x))_{x\in B(x_1,r)}, (Z(y))_{y\in B(x_2,r)})$$ and appropriately modified field $((Z(x) + \Omega_{\frac{\delta}{T}})_{x\in B(x_1,r)}, (W(y) + \Omega_{\frac{\delta}T})_{y\in B(x_2,r)})$, where $(W(x))$ is an independent copy of $(Z(x))$ and $ \Omega_{\frac{\delta}T}$ is a centered Gaussian random variable with variance $\ln \frac{T}{\delta}$ and independent of both fields $(Z(x))$ and $(W(x))$. We observe that we always have 
\begin{align*}
	\mathbb{E}Z(x)Z(y) \leq \mathbb{E}(Z(x) + \Omega_{\frac{\delta}T})(W(y) + \Omega_{\frac{\delta}T}), \quad x,y \in B(x_1,r)\cup B(x_2,r)
\end{align*}
so that the assumption for Kahane's convexity inequality will be satisfied for our fields.
Let us also denote $I = e^{\gamma N_{\ln\frac{T}{\delta}} - \frac{\gamma^2}{2}\ln\frac{T}{\delta}}$ and let $M_W$ be a GMC built on the field $(W(x))$.

Using \eqref{e:a} and the Kahane convexity inequality (see e.g. Proposition 6.1 in \cite{Aru}) 
we find that (actually to use Kahane's theorem we should apply the inequality for mollifiers, take the limit as $\epsilon\rightarrow 0$ and conclude from the Lebesgue dominated convergence) 
\begin{align*}
	\mathbb{E}e^{-tM(B(x_1,r)) - tM(B(x_2,r))} &\leq \mathbb{E}e^{-tI\big(M(B(x_1,r)) + M_W(B(x_2,r))\big)}\\ 
	&= \mathbb{E}e^{-tI\big(M(B(x_1,r)) + M_W(B(x_2,r))\big)}\Big(1_{\{I < t^{-0.2}\}} + 1_{\{I \geq  t^{-0.2}\}}\Big)\\
	&=: J_1 + J_2.
\end{align*}
Since $\Omega_{\frac{\delta}T}$ is a centered Gaussian random variable with variance $\ln \frac{T}{\delta}$ we have 
\begin{align*}
	J_1 \leq \mathbb{P}(I < t^{-0.2}) \leq Ce^{-c(\ln t)^2},
\end{align*}
for two positive constants $c, C$. For $J_2$ we write 
\begin{align*}
  J_2 \leq \mathbb{E}e^{-t^{0.8}\big(M(B(x_1,r)) + M_W(B(x_2,r))\big)} = \Big(\mathbb{E}e^{-t^{0.8}M(B(0,r))}\Big)^2,
\end{align*}	
where we used the assumption that $(W(x))$ is an independent copy of $(Z(x))$ and stationarity of GMC. Now, by scaling property
\begin{align*}
\mathbb{E}e^{-t^{0.8}M(B(0,r))} = \mathbb{E}e^{-t^{0.8}\widehat{I}M_1},
\end{align*}
where $\widehat{I} = r^de^{\gamma\Omega_r - \frac{\gamma^2}{2}\mathbb{E}\Omega^2_r}$ and $\Omega_r$ is a centered Gaussian random variable with variance $\ln\frac{1}{r}$ independent of $M_1$. We repeat the above split to conclude that
\begin{align*}
	\mathbb{E}e^{-t^{0.8}\widehat{I}M_1} \leq C_1e^{-c_1(\ln t)^2} + \mathbb{E}e^{-t^{0.64}M_1},
\end{align*}
for some positive constants $c_1, C_1$.
Putting the things together we find two positive constants $c_0, C_0$ such that  
\begin{align*}
	\mathbb{E}e^{-tM_1}  &\leq J_1 + J_2 \leq  Ce^{-c(\ln t)^2} + \Big(C_1e^{-c_1(\ln t)^2} + \mathbb{E}e^{-t^{0.64}M_1}\Big)^2\\
	&\leq C_0e^{-c_0(\ln t)^2} + \Big(\mathbb{E}e^{-t^{0.64}M_1}\Big)^2.
\end{align*}
If $t$ is sufficiently large, by decreasing $c_0$ we may assume that $C_0 = 1$. Now define $p(\alpha) = \mathbb{E}e^{-e^{\alpha} M_1}$, $\alpha \geq 0$. If $t= e^s$ we have already obtained inequality
\begin{align*}
	p(s) \leq e^{-c_0s^2} + p^2(0.64 s).
\end{align*}
As a result we get exactly the same relation as in the final step of proof in \cite[Lem.4.3]{Aru}. By repeating in exactly the same way his argument we conclude \eqref{e:8a}. For the second assertion we use Chebyshev's inequality to get
\begin{align*}
	\mathbb{P}(M_1 \leq \epsilon) =P(e^{-\frac 1\varepsilon M_1}\ge e^{-1})\leq e\E e^{-\frac 1{\varepsilon}M_1}.
\end{align*}
Now \eqref{e:9} follows directly from \eqref{e:8a}.
\end{proof}

\begin{proof} (of Proposition \ref{lowar})
Let $r\in (0,1)$. Rewriting 
Lemma \ref{L} in the notation \eqref{e:ar} we have
\begin{align}
	Q_r(x) =& \mathbb{E}e^{-e^xM_{\gamma}(B(0,r))} = C(r)e^{-a(r)x^2 + bx}\mathbb{E}\Big(T_1^{-1-b}M_1^be^{2xa(r)\ln\frac{T_1}{M_1}- a(r)\big(\ln\frac{T_1}{M_1}\big)^2}\Big)\label{e:QrT}\\
	&=: J_1(x)J_2(x),\qquad x>0,\label{e:QrJ}
\end{align}
where
\begin{align}
	J_1(x) = C(r)e^{-a(r)x^2 + bx}, \quad J_2(x) = \mathbb{E}\Big(T_1^{-1-b}M_1^be^{2xa(r)\ln\frac{T_1}{M_1}- a(r)\big(\ln\frac{T_1}{M_1}\big)^2}\Big).
\label{e:J1J2}
	\end{align}
For $\delta > 0$ we have
\begin{align*}
	J_2(x) &\geq \mathbb{E}1_{\big\{\big(\ln\frac{T_1}{M_1}\big)^2\leq \delta^2\big\}}\Big(T_1^{-1-b}M_1^be^{2xa(r)\ln\frac{T_1}{M_1}- a(r)\big(\ln\frac{T_1}{M_1}\big)^2}\Big)\\
	&\geq e^{-a(r)\delta^2}\mathbb{E}1_{\big\{2(\ln T_1)^2 + 2(\ln M_1)^2\leq \delta^2\big\}}\Big(T_1^{-1-b}M_1^be^{2xa(r)\ln\frac{T_1}{M_1}}\Big)\\
	&\geq e^{-a(r)\delta^2}\mathbb{E}1_{\big\{4(\ln T_1)^2\leq \delta^2\big\}}1_{\big\{4(\ln M_1)^2\leq \delta^2\big\}}\Big(T_1^{-1-b}M_1^be^{2xa(r)\ln\frac{T_1}{M_1}}\Big).
\end{align*}
By independence of $M_1$ and $T_1$ 
\begin{align*}
	J_2(x)=e^{-a(r)\delta^2}\mathbb{E}\Big(1_{\big\{4(\ln T_1)^2\leq \delta^2\big\}}T_1^{2xa(r)-1-b}\Big)\mathbb{E}\Big(1_{\big\{4(\ln M_1)^2\leq \delta^2\big\}}M_1^{b-2xa(r)}\Big).
\end{align*}
We estimate the components of the last equality
\begin{align*}
	\mathbb{E}\Big(1_{\big\{4(\ln T_1)^2\leq \delta^2\big\}}T_1^{2xa(r)-1-b}\Big) &= \int_{e^{-\frac{\delta}{2}}}^{e^{\frac{\delta}{2}}}t^{2xa(r)-1-b}e^{-t}dt\\
	&\geq e^{-\frac{\delta}{2}(2xa(r)-1-b)}\big(-e^{-e^{\frac{\delta}{2}}} + e^{-e^{-\frac{\delta}{2}}}\big)
\end{align*}
and
\begin{align*}
	\mathbb{E}\Big(1_{\big\{4(\ln M_1)^2\leq \delta^2\big\}}M_1^{b-2xa(r)}\Big) \geq e^{-\frac{\delta}{2}(2xa(r)-b)}\mathbb{P}(4(\ln M_1)^2\leq \delta^2).
\end{align*}
Together we get
\begin{align*}
	Q_r(x) \geq C(r)e^{-a(r)x^2 + bx}e^{-\delta(2xa(r)-b)}\mathbb{P}(4(\ln M_1)^2\leq \delta^2)\big(-e^{-e^{\frac{\delta}{2}}} + e^{-e^{-\frac{\delta}{2}}}\big).
\end{align*}
Setting $\delta = \frac{b}{ 2a(r)}$ completes the proof of the proposition up to existence of the density of $M_1$. However, for logarithmic kernel this was proved in \cite{RoVa}. Intuitively, it is well known that under mild assumptions $\mathbb{P}(M_1 > t)\approx C t^{-q}$ for positive constants $c, \ q$ and large $t$ (\cite{Wo}). So if $\delta := \frac{b}{ 4a(r)}$ is large enough, we have for any $p\in (0,1)$
\begin{align*}
	\mathbb{P}\Big(|\ln M_1|\leq \delta\Big) = \mathbb{P}\Big(M_1 \in (e^{-\delta}, e^{\delta})\Big) \geq  \mathbb{P}\Big(M_1 \in (e^{p\delta}, e^{\delta})\Big) > C(e^{-qp\delta} - e^{-q\delta}) > 0. 
\end{align*}
\end{proof}

\begin{proof} (of Theorem \ref{kappa})
Recall \eqref{e:ar}.

Then \eqref{e:QrT} can be written as
\begin{align*}
	Q_r(x) &= C(r)e^{-a(r)x^2 + bx}\mathbb{E}\Big(T_1^{-1-b}M_1^be^{2xa(r)\ln\frac{T_1}{M_1}- a(r)\big(\ln\frac{T_1}{M_1}\big)^2}\Big)\\
	& = C(r) e^{-a(r)x^2 + bx}\kappa(4a(r), 4a(r)x), \quad x > \frac{b+1}{2a(r)} ,
\end{align*}
 By a direct computation we deduce from (\ref{kap}) that for $s>0$ and $x>2b+2$ we have
\begin{align}\label{sder}
	\frac{\partial \kappa(s,x)}{\partial s} =& -\frac{1}{4}\mathbb{E}\frac{1}{M_1}\Big(\Big(\ln\frac{T_1}{M_1}\Big)^2e^{(\frac{x}{2}- b - 1)\ln\frac{T_1}{M_1}- \frac{s}{4}\big(\ln\frac{T_1}{M_1}\big)^2}\Big),\\
	\frac{\partial \kappa(s,x)}{\partial x} =& \frac12\mathbb{E}\Big(\frac{1}{M_1}\ln\frac{T_1}{M_1}e^{(\frac{x}{2}- b - 1)\ln\frac{T_1}{M_1}- \frac{s}{4}\big(\ln\frac{T_1}{M_1}\big)^2}\Big),\label{e:xder}\\
	\frac{\partial^2 \kappa(s,x)}{\partial x^2} =& \frac14\mathbb{E}\Big(\frac{1}{M_1}\left(\ln\frac{T_1}{M_1}\right)^2e^{(\frac{x}{2}- b - 1)\ln\frac{T_1}{M_1}- \frac{s}{4}\big(\ln\frac{T_1}{M_1}\big)^2}\Big).\label{e:x2der}
\end{align}
Hence $\kappa$ satisfies \eqref{bckwd}.

Taking the associated limits in \eqref{kap} we deduce the boundary behavior 
\begin{align*}
	\kappa(s,2b) &= \mathbb{E}\Big(\frac{1}{M_1}e^{-\ln\frac{T_1}{M_1}- \frac{s}{4}\big(\ln\frac{T_1}{M_1}\big)^2}\Big), \quad \kappa(0,z) = \Gamma\Big(\frac{z}{2}- b\Big)\mathbb{E}M_1^{b - \frac{z}{2}},\\
	\frac{\partial\kappa(s,2b)}{\partial x} &=  \frac12\mathbb{E}\Big(\frac{1}{M_1}\ln\frac{T_1}{M_1}e^{-\ln\frac{T_1}{M_1}- \frac{s}{4}\big(\ln\frac{T_1}{M_1}\big)^2}\Big).
\end{align*}
We notice, that the backward heat equation (\ref{bckwd}) can be reduced to ODE of the second order. Indeed, let $\lambda > 0$ and define
\begin{align*}
	L_{\lambda}\kappa(z) = \int_0^{\infty}e^{-\lambda s}\kappa(s,z)ds.
\end{align*}
The Laplace transform is well defined since $s\mapsto \kappa(s,z)$ is bounded for any fixed $z> 2b$. Equation (\ref{bckwd}), integration by parts and arguments for differentiation under the integral enforces that
\begin{align*}
	L^{''}_{\lambda}\kappa(z) + \lambda L_{\lambda}\kappa(z) = \kappa(0,z). 
\end{align*}

We next verify that $\kappa$ is in postulated class.
By \eqref{sder} we see that $\kappa$ is decreasing in $s$.  
For any fixed $x > 2b$ we have $\kappa(s,x) \leq \Gamma\Big(\frac{x}{2}- b\Big)\mathbb{E}M_1^{b-\frac x2}$ so by  Lebesgue dominated convergence $\kappa$ decreases to $0$ as $s\rightarrow\infty$.
Differentiating \eqref{sder} again with respect to $s$-variable we see that the result is nonnegative, thus proving convexity of $\kappa$ with respect to variable $s$. 

The convexity in the second variable follows form \eqref{e:x2der}. To prove condition 3. we notice, that by \eqref{sder} and \eqref{e:xder} we have 
\begin{align*}
	\frac{\partial \kappa(s,x)}{\partial s} + \frac{\partial \kappa(s,x)}{\partial x} &= \mathbb{E}\Big(\Big[-\frac{1}{4}\Big(\ln\frac{T_1}{M_1}\Big)^2 + \frac12 \ln\frac{T_1}{M_1}\Big]\frac{1}{M_1}e^{(\frac{x}{2}- b - 1)\ln\frac{T_1}{M_1}- \frac{s}{4}\big(\ln\frac{T_1}{M_1}\big)^2}\Big)\\
	&= -\frac14\mathbb{E}\Big(\Big(\ln\frac{T_1}{M_1} - 1\Big)^2\frac{1}{M_1}e^{(\frac{x}{2}- b - 1)\ln\frac{T_1}{M_1}- \frac{s}{4}\big(\ln\frac{T_1}{M_1}\big)^2}\Big)\\
	&\quad + \frac14 \mathbb{E}\Big(\frac{1}{M_1}e^{(\frac{x}{2}- b - 1)\ln\frac{T_1}{M_1}- \frac{s}{4}\big(\ln\frac{T_1}{M_1}\big)^2}\Big) \leq \frac14\kappa(s,x). 
\end{align*}
\ \\
The last part of the proof is devoted to show property in point 4, proving at the same time   conclusion (\ref{oda1}). We  prove first \eqref{oda1} for $p=2$.
Then we will show how to adapt this proof for smaller values of  $p$.

We first investigate the order of convergence of random variable $\ln\frac{T_1}{M_1}$. Let $t\geq 1$ (which will be then set to be large). We have 
\begin{align}
	\mathbb{P}\Big(\ln\frac{T_1}{M_1} > t\Big) = \mathbb{P}\Big(T_1 > M_1 e^{t}\Big) = \mathbb{E}e^{-e^tM_1}. 
\label{e:logtail}
	\end{align}
Hence, by \eqref{e:c0} for $t\ge t_0$ we have
 \begin{align}\label{fstlm}
 	\mathbb{P}\Big(\ln\frac{T_1}{M_1} > t\Big)\leq C_0e^{-c_0 t^2}.
 \end{align}
We split $\kappa(s,x)$ into two summands.  Namely let $D = \{\ln\frac{T_1}{M_1}> t_0\}$ and assume that $x > 2(b+1)$. We have
\begin{align}
	\kappa(s,x) = \mathbb{E}\Big(1_D + 1_{D'}\Big)\Big(\frac{1}{M_1}e^{(\frac{x}{2}- b - 1)\ln\frac{T_1}{M_1}- \frac{s}{4}\big(\ln\frac{T_1}{M_1}\big)^2}\Big) =: I + II.
	\label{e:k_decomposition}
\end{align} 
For the component $II$ we first apply Cauchy-Schwarz inequality
\begin{align*}
	II \leq \sqrt{\mathbb{E}\Big(\frac{1}{M^2_1}\Big)}\sqrt{\mathbb{E}\Big(1_{D'}e^{2(\frac{x}{2}- b - 1)\ln\frac{T_1}{M_1}- \frac{s}{2}\big(\ln\frac{T_1}{M_1}\big)^2}\Big)}.
\end{align*}
By the well known finiteness of negative moments of GMC (cf. Theorem \ref{Qor}) the first expectation on the RHS of the last inequality is finite. 
We then focus on the expectation under the second square root. We have
\begin{equation*}
 \mathbb{E}\Big(1_{D'}e^{2(\frac{x}{2}- b - 1)\ln\frac{T_1}{M_1}- \frac{s}{2}\big(\ln\frac{T_1}{M_1}\big)^2}\Big)\le e^{(x-2b-2){t_0}}.
\end{equation*}
Hence 
\begin{equation}
II\le C_1 e^{\frac 12 t_0 x}
\label{e:II}
\end{equation}
for some positive constant $C_1$.
For the component $I$ we again use Cauchy-Schwarz inequality obtaining
\begin{align}
	I \leq \sqrt{\mathbb{E}\Big(\frac{1}{M^2_1}\Big)}\sqrt{\mathbb{E}\Big(1_De^{2(\frac{x}{2}- b - 1)\ln\frac{T_1}{M_1}- \frac{s}{2}\big(\ln\frac{T_1}{M_1}\big)^2}\Big)}.
\label{e:ICauchy}
	\end{align}
We again use the finiteness of negative moments of GMC to see that the first expectation on the right hand side is finite. We now investigate the second expectation on the right.
 We have
\begin{align*}
	\mathbb{E}\Big(1_De^{2(\frac{x}{2}- b - 1)\ln\frac{T_1}{M_1}- \frac{s}{2}\big(\ln\frac{T_1}{M_1}\big)^2}\Big) &= \int_{t_0}^{\infty}e^{2(\frac{x}{2}- b - 1)z- \frac{s}{2}z^2}\mathbb{P}\Big(\ln\frac{T_1}{M_1}\in dz\Big)=:J_1
\end{align*}
Recall the elementary fact for a random variable $X$ without atoms and appropriately smooth and integrable function $\phi(\cdot)$
\begin{align}\label{pomfi}
	\mathbb{E}\phi(X)1_{\{X > t\}} = \phi(t)\mathbb{P}(X > t) + \int_t^{\infty}\phi'(z)\mathbb{P}(X > z)dz.
\end{align}
Applying the above for $X = \ln\frac{T_1}{M_1}$,  $\phi(z) = e^{2(\frac{x}{2}- b - 1)z- \frac{s}{2}z^2}$ and $t=t_0$ we get
\begin{align}
	J_1 \leq & e^{-\frac{st_0^2}{2} + (x - 2b -2)t_0}\mathbb{P}\Big( \ln\frac{T_1}{M_1} > t_0\Big)\notag \\
	&- \int_{t_0}^{\infty}e^{({x}- 2b - 2)z- \frac{s}{2}z^2}(sz - x + 2b + 2)\mathbb P(\ln \frac {T_1}{M_1}>z)dz\label{e:temp}\\
	\le & e^{t_0 (x-2b-2)}
	+e^x\int_{t_0}^{t_0\vee \frac {x-2b-2}{s}} 
	e^{(x-2b-2)z-\frac s2 z^2}\mathbb P(\ln \frac {T_1}{M_1}>z)dz,\label{e:J1}
\end{align}
Where in the last estimate we have used the fact that the term under the integral in \eqref{e:temp} is positive for $z>\frac{x-2b-2}{s}$
and for $z\le\frac{x-2b-2}{s}$ we use $x-2b-2-sz\le e^x $. We estimate the integral in  \eqref{e:J1}  using \eqref{e:logtail} and 
\eqref{e:c0}
and then extending the range of integration:
\begin{align*}
 \int_{t_0}^{t_0\vee \frac {x-2b-2}{s}} &
	e^{(x-2b-2)z-\frac s2 z^2}\mathbb P(\ln \frac {T_1}{M_1}>z)dz\\
	\le &C_0 \int_{\mathbb R}
	e^{(x-2b-2)z-\frac s2 z^2} e^{-c_0z^2}dz\\
=&C_0 e^{\frac{(x-2b-2)^2}{4(\frac s2+c_0)}}\int_{\mathbb R}
e^{-\left(\sqrt{\frac s2+c_0}z-\frac{x-2b-2}{2\sqrt{\frac s2+ c_0}} \right)^2}dz\\
\le & \frac{C_0\sqrt \pi}{\sqrt{\frac s2+c_0}}e^{\frac {x^2}{2s+4c_0}}.
\end{align*}
Thus, by \eqref{e:ICauchy} and \eqref{e:J1} we have
\begin{equation}
 I\le C_2 e^{\frac{t_0}{2}x}+\frac{C_3}{\sqrt{2c_0}}e^{\frac x2} e^{\frac {x^2}{4(s+2c_0)}}
 \label{e:I_est}
\end{equation}
for some positive constants $C_2$ and $C_3$.

Combining \eqref{e:I_est}, \eqref{e:II} and \eqref{e:k_decomposition} we obtain \eqref{e:kappa_est1} for $p=2$.

\ \\
Finally we prove \eqref{oda1} for $p=2$: The lower estimate  follows from Proposition \ref{lowar}. To prove the upper estimate we use \eqref{e:Qr} and \eqref{e:kappa_est1} for $p=2$ to obtain 
\begin{align*}
Q_r(x)\le 	C(r) e^{-a(r)x^2 + bx}L_2^{4a(r)x}
	e^{\frac{(4a(r))^2x^2}{4(4a(r)+2c_0)}}.
\end{align*}
A straightforward computation gives 
\begin{align*}
	-a(r)x^2 -C_1(2,r)x\leq \ln Q_r(x) \leq -\Big[a(r) - \frac{4a^2(r)}{4a(r) + 2c_0}\Big]x^2 + C_1(2,r)x,
\end{align*}
for some $C_1(2,r)>0$ and $x\ge x_0(2,r)$,
which finishes the proof. 

\medskip

The proof of \eqref{oda1} in the case  $p\in (1,2)$ is only a slight modification of the one given above for $p=2$. 

Indeed, the crucial observation here is to replace the Cauchy Schwartz inequality in \ref{e:ICauchy} with H\"{o}lder inequality. Namely
\begin{align}
	I \leq \Big(\mathbb{E}\Big(M^{-\frac{p}{p-1}}_1\Big)\Big)^{1 - \frac{1}{p}}\Big(\mathbb{E}\Big(1_De^{p(\frac{x}{2}- b - 1)\ln\frac{T_1}{M_1}- p\frac{s}{4}\big(\ln\frac{T_1}{M_1}\big)^2}\Big)\Big)^{\frac{1}{p}}.
\label{e:IHol}
	\end{align}
	Every negative moment exists so we repeat exactly the same reasoning with estimations on the second component changing only $2$ to the constant $p$ in the exponent. Thus for the two constants $A = A(t_0,p)$ and $B = B(t_0,p)$ we get
\begin{equation}
 I\le C_2 A^x+C_3 B^xe^{\frac{x^2p}{4(ps + 4c_0)}}.
 \label{e:I_estHol}
\end{equation}
for some positive constants $C_2$ and $C_3$ and in exactly same way we obtain (\ref{e:kappa_est1}). Finally, we apply the new bound repeating what we did proving (\ref{oda1}). A straightforward computation gives 
\begin{multline*}
	-a(r)x^2 -C_1(r,p)x-C_2(r,p)\leq \ln Q_r(x) \\
	\leq -\Big[a(r) - \frac{16a^2(r)p}{16pa(r) + 16 c_0}\Big]x^2 + C_1(r,p)x+C_2(r,p),
\end{multline*}
for some $C_1(r,p), C_2(r,p)>0$ and $x\ge x_0(r,p)$ 
which implies \eqref{oda1}.
\end{proof}

\begin{proof}(of Theorem \ref{INLTB}) The proof of identity (\ref{ILTB}) is based on use of scaling property analogously to the proof of Lemma \ref{L}. To prove that (\ref{kappa2}) solves the Cauchy problem (\ref{2PDE}) we repeat the reasoning from the proof of Theorem \ref{kappa}. We omit the technical details.
Finally, claim (\ref{lnka}) follows by equivalence $\lambda\rightarrow\infty$
\begin{align*}
	\mathbb{E}e^{-\frac{\lambda}{M_r}} \approx C_1\lambda^{-q} \quad \Leftrightarrow \quad \mathbb{P}(M_r > \lambda) \approx C_2\lambda^{-q}
\end{align*}
for some positive constants $C_1, C_2$ (see \cite[Corr.4]{Wo}) and Theorem \cite[Thm.1]{Wo}. 
\end{proof}

 \begin{proof}(of Theorem \ref{thm:cr_bound}) The first claim is a direct consequence of \eqref{oda1}. For the second claim we apply the criterion from Lemma \ref{mean} which we know, from the proof of Theorem \ref{mean0}, to work for $d\geq 3$.  Using then (\ref{c1u}) we conclude the bound given by (\ref{uci}).  
 \end{proof}

\begin{proof} (of Theorem \ref{cezero})   

\ \\
\emph{Part I}\\
  We will modify the idea used in the proof of Theorem \ref{Qor}. Namely, following the reasoning from the beginning of the proof of \ref{Qor} we consider 
  two balls $B_R$ and $B_r$ of radius $R$ and $r$, respectively, 
  contained in $B(0,1)$ and
  whose distance from each other is at least $T$. To make it possible we have to assume that 
$R+r<1-\frac T2$.
  We will use Kahane's convexity inequality. Recall that $\mathbb{E}Z(x)Z(y) = \ln\frac{T}{|x-y|}$.
	
	\ \\
	\emph{Case 1.}\\
	Assume that $T\leq 2$. Then observe, that the Gaussian field $((Z(x))_{x\in B_R}, (Z(x))_{x\in B_r})$ is dominated by 
  $((Z(x))_{x\in B_R}, (W(x))_{x\in B_r})$, where
  $W$ is an independent copy of $Z$. By Kahane's convexity inequality we obtain
  \begin{equation*}
   Q_1(x)\le \E e^{-e^x(M_{B_R}+M_{B_r})}\le \E e^{-e^{x}M_{B_R}}\E e^{-e^{x}M_{B_r}}.
  \end{equation*}
Hence, by  \eqref{e:cr_estim1}, we obtain
\begin{equation*}
 \overline c_1\ge \frac{\overline c_1a(R)}{\overline c_1 + a(R)} + \frac{\overline c_1a(r)}{\overline c_1 + a(r)}.
\end{equation*}
We set $R=r$ and let $r\to \frac 12 -\frac T4$.
A simple algebra shows that if $\overline c_1<\infty$ then 
\eqref{cees3}
follows.  \eqref{cees} follows directly from \eqref{cees3}.

\ \\
\emph{Case 2.}\\
Assume that $T > 1$ and let $B_R$ and $B_r$  be balls of radius $R$ and $r$, respectively, 
  contained in $B(0,1)$ and at a distance $1$ from each other.  Then the Gaussian field $((Z(x))_{x\in B_R}, (Z(x))_{x\in B_r})$ is dominated by $((Z(x) + \Omega_{\ln T})_{x\in B_R}, (W(x) + \Omega_{\ln T})_{x\in B_r})$, where $\Omega_{\ln T}$ is a centered Gaussian random variable with variance $\ln T$ and independent of both $Z$ and $W$. 
By Kahane's convexity inequality we obtain
  \begin{align*}
   Q_1(x)&\le \E e^{-e^xe^{\gamma\Omega_{\ln T} - \frac{\gamma^2}{2}\ln T}(M_{B_R}+\tilde M_{B_r})}\\
	&=  \E e^{-e^{x}e^{\gamma\Omega_{\ln T} - \frac{\gamma^2}{2}\ln T}(M_{B_R}+\tilde M_{B_r})}(1_{\{\gamma\Omega_{\ln T} - \frac{\gamma^2}{2}\ln T > -px\}} + 1_{\{\gamma\Omega_{\ln T} - \frac{\gamma^2}{2}\ln T \leq -px\}}),
  \end{align*}
	where $M_R$, $\tilde M_r$ and $\Omega_{\ln T}$ are independent and $p\in (0,1)$ is arbitrary. For this reason for any $\varepsilon>0$ and sufficiently large $x$ we get
	\begin{align*}
		 Q_1(x) \leq & \E e^{-e^{(1-p)x}(M_{B_R}+\tilde M_{B_r})} + \mathbb{P}(\gamma\Omega_{\ln T} - \frac{\gamma^2}{2}\ln T \leq -px)\\
		 \le & C_1 e^{-(1-p)^2x^2(\overline c_r + \overline c_R-\varepsilon)} + C_2 e^{-\frac{(px-\frac {\gamma^2}2\ln T )^2}{2\gamma^2\ln T}}.
	\end{align*}
Hence, for every $p\in (0,1)$ we have
\begin{align*}
	\overline c_1 \geq (1-p)^2(\overline c_r + \overline c_R)\wedge \frac{p^2}{2\gamma^2\ln T}.
\end{align*}
Using again (\ref{e:cr_estim1}) we rewrite the last condition as  
\begin{align*}
	\overline c_1 \geq (1-p)^2\Big(\frac{\overline c_1a(R)}{\overline c_1 + a(R)} + \frac{\overline c_1a(r)}{\overline c_1 + a(r)}\Big)\wedge \frac{p^2}{2\gamma^2\ln T}.
\end{align*}
Setting $r = R$, we get (\ref{cees33}) after some elementary algebra. Notice, that the bound on $p$ is set so that the RHS of (\ref{cees33}) is strictly positive.  

\ \\
\emph{Part II}\\
To prove (\ref{cees2}) we consider $2d$ balls inside $B(0,1)$. We have $d$ orthogonal lines crossing $0$ and two balls with radius $r$ on each line placed such that the distance between each of the ball with radius $r$ exceeds $1$.
To be able to do it we must have 
$\sqrt 2 (1-r)\ge 1+2r$.

\ \\
\emph{Case 1.}\\ If $T\leq 1$ then the field $Z(\cdot)$ is dominated in sense of Kahane's inequality by the $2d$ independent copies of $Z(\cdot)$ on the balls with radius $r$. Such a construction leads to inequality 
\begin{align*}
	Q_1(x) \leq Q^{2d}_r(x),
\end{align*}
where $r\leq \frac {\sqrt 2-1}{2+\sqrt 2}$.
If $\overline c_1$ if finite, the last inequality together with \eqref{oda1} gives 
\begin{align*}
	1 \geq \frac{2da(r)}{a(r) + \overline c_1}.
\end{align*}
Solving the last inequality with respect to $c_1$ yields (\ref{cees2}). 

\ \\
\emph{Case 2.}\\
If $T> 1$ then we repeat reasoning from case 2, part I.   We obtain
\begin{align*}
   Q_1(x)\le \E e^{-e^xe^{\gamma\Omega_{\ln T} - \frac{\gamma^2}{2}\ln T}(\sum_{i=1}^{2d}M^{(i)}_r)},
\end{align*}	
	where $M^{(i)}_r$ are independent copies of $M_r$ and independent of $\Omega_{\ln T}$.  As before, we conclude for arbitrary $p\in (0,1)$, $\varepsilon>0$ and sufficiently large $x$
\begin{align*}
	 Q_1(x) &\le  \E e^{-e^{x}e^{\gamma\Omega_{\ln T} - \frac{\gamma^2}{2}\ln T}(M_{B_R}+M_{B_r})}(1_{\{\gamma\Omega_{\ln T} - \frac{\gamma^2}{2}\ln T > -px\}} + 1_{\{\gamma\Omega_{\ln T} - \frac{\gamma^2}{2}\ln T \leq -px\}})\\
	&\le C_1 e^{-(1-p)^2x^2 2d(\overline c_r -\varepsilon)} + C_2 e^{-\frac{p^2x^2}{2\gamma^2\ln T}},
  \end{align*}
	where $C_1, \ C_2$ are positive constants. Thus, for large $x$
	\begin{align*}
		\overline c_1 \geq 2d(1-p)^2\overline c_r \wedge \frac{p^2}{2\gamma^2\ln T}.
	\end{align*}
	Using (\ref{oda1}) we rewrite the last condition as
	\begin{align*}
		\overline c_1 \geq (2d(1-p)^2 - 1)a(r) \wedge \frac{p^2}{2\gamma^2\ln T}.
	\end{align*}
	Applying the bound on $r$ finishes the proof. 
\end{proof}

\begin{proof}(of Theorem \ref{CNQR})\\
 We again use formula \eqref{e:Qr}. Namely, 
\begin{align*}
\mathbb{E}e^{-e^x M_r} = C(r)e^{-a(r)x^2 + bx}\mathbb{E}\Big(T_1^{-1-b}M_1^be^{2xa(r)\ln\frac{T_1}{M_1}- a(r)\big(\ln\frac{T_1}{M_1}\big)^2}\Big),
\end{align*}
where $T_1$ is exponential with parameter $1$, independent of $M_1$. We have 
\begin{align*}
	&\mathbb{E}\Big(T_1^{-1-b}M_1^be^{2xa(r)\ln\frac{T_1}{M_1}- a(r)\big(\ln\frac{T_1}{M_1}\big)^2}\Big)\\
	&\qquad \qquad \geq \mathbb{E}e^{(2xa(r) - b -1)\ln T_1 - 2a(r)(\ln T_1)^2} \mathbb{E}e^{-(2xa(r) - b)\ln M_1 - 2a(r)(\ln M_1)^2}.
\end{align*}
The first component of the last expression is easy to deal with. Indeed, for $x$ such that $2xa(r) - b -1 \geq 0$
\begin{align*}
	\mathbb{E}e^{(2xa(r) - b -1)\ln T_1 - 2a(r)(\ln T_1)^2} &\geq \mathbb{E}e^{(2xa(r) - b -1)\ln T_1 - 2a(r)(\ln T_1)^2}1_{\{T_1 > 1\}}\\
	&\geq \mathbb{E}e^{- 2a(r)(\ln T_1)^2}1_{\{T_1 > 1\}},
\end{align*}
which is a positive constant dependent on $r$. To deal with the second component, we write 
\begin{align*}
	\mathbb{E}e^{-(2xa(r) - b)\ln M_1 - 2a(r)(\ln M_1)^2} &= \mathbb{E}e^{\frac{(2xa(r) - b)^2}{8a(r)} - 2a(r)\Big(\ln M_1 + \frac{2a(r)x - b}{4a(r)} \Big)^2 }\\
	&= e^{\frac{(2xa(r) - b)^2}{8a(r)}}\mathbb{P}\Big(T_1 \geq 2a(r)\Big(\ln M_1 + \frac{2a(r)x - b}{4a(r)} \Big)^2\Big).
\end{align*}
Furthermore, for $p>0$
\begin{align*}
	\mathbb{P}&\Big(T_1 \geq 2a(r)\Big(\ln M_1 + \frac{2a(r)x - b}{4a(r)} \Big)^2\Big)\\
	&\qquad  \geq \mathbb{P}(T_1 \geq px^2)\mathbb{P}\Big(px^2 \geq 2a(r)\Big(\ln M_1 + \frac{2a(r)x - b}{4a(r)} \Big)^2\Big)\\
	&\qquad  = e^{-px^2}\mathbb{P}\Big(px^2 \geq 2a(r)\Big(\ln M_1 + \frac{2a(r)x - b}{4a(r)} \Big)^2\Big)\\
	&\qquad  = e^{-px^2}\mathbb{P}\Big(-x\Big(\sqrt{\frac{p}{2a(r)}} + \frac12\Big) + \frac{b}{4a(r)}\leq \ln M_1 \leq x\Big(\sqrt{\frac{p}{2a(r)}} - \frac12\Big) + \frac{b}{4a(r)}\Big)\\
	&\qquad = e^{-px^2}\mathbb{P}\Big(e^{-x(\sqrt{\frac{p}{2a(r)}} + \frac12)} \leq e^{-\frac{b}{4a(r)}}M_1 \leq e^{x(\sqrt{\frac{p}{2a(r)}} - \frac12)}\Big).
\end{align*}
Hence, for a positive constant $\widehat{C}(r)$ dependent on $r$ and arbitrary $p>0$ 
\begin{align*}
	\mathbb{E}e^{-e^xM_r} \geq \widehat{C}(r)e^{-(p + \frac{a(r)}{2})x^2 + \frac{b}{2}x + \frac{b^2}{8a(r)}}\mathbb{P}\Big(e^{-x(\sqrt{\frac{p}{2a(r)}} + \frac12)} \leq e^{-\frac{b}{4a(r)}}M_1 \leq e^{x(\sqrt{\frac{p}{2a(r)}} - \frac12)}\Big).
\end{align*}
Set $p = \frac12a(r)q^2$, $q\in (0,1)$. Then, for a positive constant $\widehat{C}_1(r)$ 
\begin{multline*}
	\mathbb{E}e^{-e^xM_r} + \widehat C_1(r) e^{-\frac{a(r)}{2}(q^2 + 1)x^2} \mathbb{P}\Big(e^{-\frac{b}{4a(r)}}M_1 \leq e^{-x\frac{1+q}{2}}\Big) \\ \geq \widehat{C}_1(r)e^{-\frac{a(r)}{2}(q^2 + 1)x^2}\mathbb{P}\Big(e^{-\frac{b}{4a(r)}}M_1 \leq e^{-x\frac{1-q}{2}} \Big).
\end{multline*}
Using the definition of $\overline c_r$ and Lemma 
\ref{lem:lapl_distrib} (a) we obtain that for any $\varepsilon, \delta>0$ there exists  $C_2(r,\delta)>0$,  $x\ge x_0(r,\delta)$ we have
\begin{align*}
e^{-x^2(\overline c_r -\delta-\frac{a(r)}{2}(q^2 + 1))} +e^{-x^2\frac{(1+q)^2}{4}(\overline c_1-\delta)} \geq C_2(r,\delta) e^{-x^2\frac{(1-q)^2}{4}(\overline c_1+\delta)}.
\end{align*}
Comparing the terms in the exponents, and then letting $\delta\to 0$ we 
conclude that, for arbitrary $q\in (0,1)$ such that $\overline c_r -\frac{a(r)}{2}(q^2 + 1) \geq 0$ we have necessarily
\begin{align*}
\frac{(1-q)^2}{4}\overline c_1 \geq \overline c_r -\frac{a(r)}{2}(q^2 + 1).
\end{align*}
\end{proof}

\ \\
\textbf{Data Availability Statement} Data sharing not applicable to this article as no data sets were generated or
analysed during the current study.

\bibliographystyle{amsplain}

\end{document}